\documentclass[10pt, twoside, reqno]{amsart}

\usepackage{amsmath, amsthm, amssymb, color, graphics}
\usepackage[english]{babel}
\usepackage{amsfonts}
\usepackage{bbm}
\usepackage{verbatim}
\usepackage{microtype}

\usepackage[headsep=30pt,footskip=25pt, left = 1.2in,  right = 1.2in, bottom=1.3in,
top=1.3in]{geometry}

\usepackage[utf8]{inputenc} 
\usepackage{mathrsfs} 
\usepackage[autostyle]{csquotes} 
\usepackage{tikz-cd} 
\usepackage{mathtools} 
\usepackage{cite} 
\usepackage{colonequals} 
\usepackage[linktocpage]{hyperref}

\newcommand{\Z}{\mathbb{Z}}

\def\rar{\rightarrow}

\def\CO{{\mathcal{O}}}

\theoremstyle{plain}
\newtheorem{theorem}{Theorem}[section]
\newtheorem{proposition}[theorem]{Proposition}
\newtheorem{lemma}[theorem]{Lemma}

\newtheorem{definition}[theorem]{Definition}
\newtheorem{example}[theorem]{Example}

\theoremstyle{remark}
\newtheorem{remark}[theorem]{Remark}

\newcommand{\aisle}[1]{\ensuremath{\mathcal{#1} ^{\leq 0}}}

\newcommand{\cat}[1]{\ensuremath{\mathcal{#1}}}

\newcommand{\deriveq}[1]{\ensuremath{\mathbf{D} (\mathrm{Qcoh\hspace{1mm}} #1)}}

\newcommand{\deriveb}{\ensuremath{\mathbf{D}^b (\mathrm{Coh \hspace{1mm}} \mathbb{P}^1_k)}}

\newcommand{\Ext}{\ensuremath{\mathrm{Ext}}}

\newcommand{\Hom}{\ensuremath{\mathrm{Hom}}}

\newcommand{\perfect}[1]{\ensuremath {\mathrm{{Perf}}{(#1)}}}

\newcommand{\preaisle}[1]{\ensuremath{\langle \mathcal{#1} \rangle ^{\leq 0} }}

\newcommand{\spec}[1]{\ensuremath{\mathrm{Spec } #1}} 

\newcommand{\Supp}{\ensuremath{\mathrm{Supp}}}

\newcommand{\tensor}{\ensuremath{\otimes}-}

\title[Tensor weight structures and t-structures on derived categories]{Tensor weight structures and t-structures on derived categories of Noetherian schemes}
\date{} 
\dedicatory{}

\author[Umesh V  Dubey]{Umesh V Dubey}
\address{Harish-Chandra Research Institute,  A CI of Homi Bhabha National Institute,  Chhatnag 
Road,  Jhunsi,  Prayagraj  211019,  India.}
\curraddr{}
\email{umeshdubey@hri.res.in}
\thanks{}

\author[Gopinath Sahoo]{Gopinath Sahoo}

\address{Harish-Chandra Research Institute,  A CI of Homi Bhabha National Institute,  Chhatnag Road,  Jhunsi,  Prayagraj  211019,  India.}
\curraddr{}
\email{gopinathsahoo@hri.res.in}
\thanks{}


\begin{document}

\begin{abstract}
 	We give a condition which characterises those weight structures on a derived category 
 	which come from a Thomason filtration on the underlying scheme.  Weight 
 	structures satisfying our condition will be called $\otimes ^c$-weight structures. More 
 	precisely,  for a Noetherian separated scheme $X$, we give a bijection between the set of 
 	compactly generated $\otimes ^c$-weight structures on $\deriveq X$ and the set of 
 	Thomason filtrations of $X$.  We achieve this classification in two steps. First, we 
 	show that the bijection \cite[Theorem 4.10]{SP16} restricts to give a bijection between the 
 	set of compactly generated $\otimes ^c$-weight structures and the set of compactly 
 	generated tensor t-structures. We then use our earlier classification of compactly 
 	generated tensor t-structures to obtain the desired result.  We also study some immediate 
 	consequences of these classifications in the particular case of the projective line.  We show 
 	that in contrast to the case of tensor t-structures,  there are no non-trivial tensor weight 
 	structures on $\deriveb$. 
\end{abstract}

\maketitle



\section{Introduction}

Weight structures on triangulated categories were introduced by Bondarko \cite{Bon10} as an 
important natural counterpart of t-structures with applications to Voevodsky's category of
motives.   Pauksztello independently came up with the same notion while trying to obtain a dual 
version of a result due to Hoshino,  Kato and Miyachi; he termed it co-t-structures,  see 
\cite{Pau08}.  It has been observed by Bondarko that the two notions, t-structures and weight 
structures, are connected by interesting relations.  In this vein,  S\v{t}ov\'{\i}\v{c}ek and 
Posp\'{\i}\v{s}il have proved for a certain class of triangulated categories, the collection of 
compactly generated t-structures and compactly generated weight structures are in bijection 
\cite[Theorem 4.10]{SP16} with each other, where the bijection goes via a duality at the 
compact level.  In particular, this bijection holds in the derived category of a Noetherian ring 
$R$ and  since in this case, we have the classification of compactly generated t-structures in 
terms of Thomason filtrations of $\spec R$ \cite[Theorem 3.11]{AJS10},  they obtain 
a classification of compactly generated weight structures of $\mathbf{D}(R)$.

Our aim in this short article is twofold: first to generalize the theorem of {S}\v{t}ov\'{\i}\v{c}ek 
and Posp\'{\i}\v{s}il  \cite[Theorem 4.15]{SP16} to the case of separated Noetherian schemes,  
and second to understand the two types of notions, in the simplest 
non-affine situation - the derived category of the projective line over a field $k$.  Our
interest in this special case arose partly from the work of Krause and Stevenson \cite{KS19},  
where the authors study the localizing subcategories of $\mathbf{D}(\mathrm{Qcoh 
\hspace{1mm}} \mathbb{P} ^1 _k)$, and partly from our desire to better understand the
general results.

In our earlier work \cite{DS22},  we have shown that a t-structure on $\deriveq X$ supported on 
a Thomason filtration of a Noetherian scheme $X$ satisfies a tensor condition.  We call 
them tensor t-structures.   In this article, we introduce the analogous notion of tensor weight 
structures, also a slightly weaker notion which we call $\otimes ^c$-weight structures.  We 
then show that the bijection \cite[Theorem 4.10]{SP16} restricts to a bijection between tensor 
t-structures and $\otimes ^c$-weight structures; this can be seen as a consequence of our 
Lemma \ref{preaisle and copreaisle} and Theorem \ref{t-structure}.  Next, we specialize to the 
case of the derived categories of separated Noetherian schemes and classify compactly 
generated $\otimes ^c$-weight structures in this case,  see Theorem \ref{tensor weight
structure}.

In the last section, we apply all the general theory and the classification results to 
the derived category of the projective line over a field $k$.  By our Theorem \ref{t-structure} 
classifying compactly generated tensor t-structures of $\mathbf{D}
(\mathrm{Qcoh \hspace{1mm}} \mathbb{P}^1_k)$ is equivalent to classifying thick 
\tensor preaisles of  $\deriveb$,  so we restrict our attention to \tensor preaisles of $\deriveb$.  
We give a complete description of the \tensor preaisles in Proposition \ref{Main}, and  in 
Proposition \ref{aisle} we determine which of these are aisles or in other words give rise to t-
structures on $\deriveb$.  The result of Proposition \ref{aisle} is not new, it can possibly be 
deduced from \cite{Bez10}; also in \cite{Rudakov}, the authors describe the bounded 
t-structures on $\deriveb$ by using the classification of t-stabilities on $\deriveb$.  Finally,  we 
consider the same question for tensor weight structures,  and to our surprise, we discovered 
that there are no non-trivial tensor weight structures on $\deriveb$.

\section{Preliminaries}

Let \cat T be a triangulated category and $\cat T^c$  denote the full subcategory of compact
objects.  We recall the definition of t-structures which was introduced in \cite{BBD}.

\begin{definition}
	A \emph{t-structure} on \cat T is a pair of full subcategories $(\cat U, \cat V)$ satisfying the 
	following properties:
	\begin{itemize}
		\item[t1.] $\Sigma \cat U \subset \cat U$ and $\Sigma^{-1}  \cat V \subset \cat V$.
		\item[t2.] $\cat U \perp \Sigma^{-1} \cat V$.
		\item[t3.] For any $T \in \cat T$ there is a distinguished triangle
		\[ U \rar T \rar V \rar \Sigma U\]
		where $U \in \cat U$ and $V \in \Sigma^{-1} \cat V$.  We call such a triangle 
		\emph{truncation decomposition} of $T$.
	\end{itemize}
	
\end{definition}

Next, we quote the definition of weight structures from \cite{Bondarko18}.
\begin{definition} 
	A \emph{weight structure} on \cat T is a pair of 
	full subcategories $(\cat X,  \cat Y)$ satisfying the following properties:
	\begin{itemize}
		\item[w0.] \cat X and \cat Y are closed under direct summands.
		\item[w1.] $\Sigma^{-1} \cat X \subset \cat X$ and $ \Sigma \cat Y \subset \cat Y$.
		\item[w2.] $\cat X \perp \Sigma \cat Y$.
		\item[w3.] For any object $T\in \cat T$ there is a distinguished triangle
		\[ X \rar T \rar Y \rar \Sigma X \]
		where $X \in \cat X$ and $Y \in \Sigma \cat Y$.  The above triangle is called 
		a \emph{weight decomposition} of $T$.
	\end{itemize}
	
\end{definition}

Note that if $(\cat U, \cat V)$ is a t-structure on \cat T then $(\cat V,  \cat U)$ is a 
t-structure on $\cat T ^{op}$.  Similarly,  If $(\cat X, \cat Y)$ is a weight structure on \cat T then 
$(\cat Y, \cat X)$ is a weight structure on$\cat T ^{op}$.

For any subcategory \cat U of \cat T, we denote $\cat U ^{\perp}$ to be the full subcategory 
consisting of objects $B \in \cat T$ such that $\Hom (A, B) = 0$ for all $A \in \cat U$.  
Analogously we define $^{\perp} \cat U $ to be the full subcategory of objects  $B \in \cat T$ 
such that $\Hom (B, A) = 0$ for all $A \in \cat U$. 

\begin{definition}
	We say a t-structure $(\cat U,  \cat V)$ is \emph{compactly generated} if there is a set of 
	compact objects $\cat S$ such that  $\cat U = { }^{\perp}(\cat S ^{\perp})$.  A weight 
	structure $(\cat X,  \cat Y)$ is \emph{compactly generated} if there is a set of compact 
	objects $\cat S$ such that  $\cat X =  {}^{\perp}(\cat S ^{\perp})$.
\end{definition}

\begin{definition}
	A subcategory $\cat U$ of $\cat T$ is a \emph{preaisle} if it is closed under positive shifts
	and extensions.  Dually,  we say $\cat U$ is a \emph{copreaisle} of $\cat T$, if $\cat U$ is a 
	preaisle of $\cat T^{op}$.  
	
	A preaisle is called \emph{thick} if it is closed under direct
	summands.  We say a preaisle is \emph{cocomplete} if it is closed under coproducts
	in $\cat T$,  and \emph{complete} if it is closed under products.  Similarly, we define
	\emph{thick}, \emph{cocomplete} and \emph{complete} copreaisles. 
\end{definition}

For a t-structure $(\cat U,  \cat V)$ the subcategory $\cat U$ is a cocomplete preaisle
of $\cat T$,  and for a weight structure $(\cat X,  \cat Y)$ the subcategory $\cat X$ is a 
cocomplete copreaisle of $\cat T$.

We need the notion of stable derivators to formulate the next theorem but our requirement of 
the theory of derivators is the bare minimum.  We will not go into the precise lengthy definition 
here,  instead, we refer the reader to \cite[\S 2.1]{SP16} and references therein.  

\begin{theorem}[\text{ \cite[Theorem 4.5]{SP16}}]
\label{SP theorem}
	Let $\cat T = \mathbb{D}(e)$, where $\mathbb{D}$ is a stable derivator such that for each
	small category $I$,  $\mathbb{D}(I)$ has all small coproducts.  Then,
	\begin{itemize}
		\item[i.] There is a bijection between the set of compactly generated t-structures of 
		$\cat T$ and the set of thick preaisles of $\cat T^c$ given by
		
		\[ (\cat U, \cat V) \mapsto \cat U \cap \cat T^c \]
		\[ \cat P \mapsto (^{\perp} (\cat P ^{\perp}),  \Sigma \cat P ^{\perp}).\]

		\item[ii.] There is a bijection between the set of compactly generated weight structures of
		$\cat T$ and the set of thick copreaisles of $\cat T^c$ given by
		
		\[ (\cat X, \cat Y) \mapsto \cat X \cap \cat T^c \]
		\[ \cat P \mapsto (^{\perp} (\cat P ^{\perp}),  \Sigma^{-1} \cat P ^{\perp}).\]
	\end{itemize}
\end{theorem}

\section{Tensor weight and t-structures}

We recall the definition of tensor triangulated category from \cite[Definition A.2.1]{HPS97}.

\begin{definition}
	A \emph{tensor triangulated category} $(\cat T,  \otimes , \mathbf{1})$ is a triangulated 
	category with a compatible closed symmetric monoidal structure.  This means there is a 
	functor $- \otimes - \colon \cat T \times \cat T \to \cat T$ which is triangulated in both the 
	variables and satisfies certain compatibility conditions.  Moreover,  for each $B \in \cat T$ the 
	functor $- \otimes B$ has a right adjoint which we denote by $\mathscr{H}om (B, -) $.   The 
	functor $\mathscr{H}om (-,-)$ is triangulated in both the variables, and for any $A$,  $B$,  
	and $C$ in $\cat T$ we have natural isomorphisms $ \Hom ( A \otimes B ,  C) \rar 
	\Hom (A,  \mathscr{H}om (B, C))$. 
\end{definition}

\begin{definition}
\label{definition}
	Let $\cat T$ be a tensor triangulated category given with a preaisle $\aisle T$ satisfying 
	\[\aisle T \otimes \aisle T \subset \aisle T\text{ and }\mathbf{1} \in \aisle T.\]
	
	A preaisle  $\cat U$ of $\cat T$ is a \emph
	{\tensor preaisle (with respect to $\aisle T$)} if $\aisle T \otimes \cat U \subset \cat U$.  We 
	say a copreaisle $\cat X$ of 
	$\cat T$ is a \emph{\tensor copreaisle (with respect to $\aisle T$) } if 
	$\mathscr{H}om(\aisle T,  \cat X) \subset \cat X$.   A t-structure $(\cat U,  \cat V)$ is a called 
	a \emph{tensor t-structure} if $\cat U$ is a \tensor preaisle,  and a weight structure $(\cat X, 
	\cat Y)$ is a \emph{tensor weight structure} if $\cat X$ is a \tensor copreaisle. 
\end{definition}

Let \cat S $\subset$ \cat T be a class of objects.  We denote the smallest cocomplete preaisle 
containing \cat S by \preaisle S and call it the \emph{cocomplete preaisle generated by \cat S.}  
 If $\cat T$ does not have coproducts we denote $\preaisle S$ to be the smallest 
preaisle containing $\cat S$.

\begin{lemma}
\label{tensor generator 1}
	 Let \aisle {\cat T} be generated by a set of objects \cat K,  that is, \aisle {\cat T} =
	 \preaisle{\cat K}.  Then 
	\begin{itemize}
		\item[i.] a cocomplete preaisle \cat U of \cat T is a $\otimes$-preaisle if and only if 
	 	\cat K $\otimes$ \cat U $\subset \cat U$.
	 	
	 	\item[ii.] a complete copreaisle $\cat X$ of $\cat T$ is a \tensor copreaisle if and only if 
	 	$\mathscr{H}om(\cat K,  \cat X) \subset \cat X$.
\end{itemize}		 
	 
\end{lemma}

\begin{proof}
Part (i).
	Suppose \cat K $\otimes$ \cat U $\subset \cat U$.  We define  $\cat B$ = $\{ X \in \aisle T 
	\mid X \otimes \cat U \subset \cat U\}$.  Since $\cat U$ is a cocomplete preaisle we can 
	observe that $\cat B$ is also a cocomplete preaisle.  Now, by our assumption $\cat K \subset 
	\cat B$ so we get $\aisle T \subset \cat B$ which proves $\cat U$ is \tensor preaisle.  The 
	converse is immediate.  Part (ii).  Let $\cat B = \{ X \in \aisle T \mid \mathscr{H}om ( X ,  \cat X)  
	\subset \cat X\}$.  Since $\cat X$ is a copreaisle we can see that $\cat B$ is a preaisle.
	 Completeness of $\cat X$ implies $\cat B$ is cocomplete.  Now,  by following
	 a similar argument as in (i) we get $\cat X$ is \tensor copreaisle.
	
\end{proof}

An immediate consequence of the above lemma is  if $\aisle T = \langle \mathbf{1} \rangle 
^{\leq 0}$ then every cocomplete preaisle of $\cat T$ is a \tensor preaisle and  every complete 
copreaisle of $\cat T$ is a \tensor copreaisle.  In particular,  for a commutative ring $R$, all
cocomplete preaisle and complete copreaisle of $\mathbf{D}(R)$ satisfy the tensor condition.

\begin{definition}
We say an object $X \in \cat T$ is \emph{rigid} or \emph{strongly dualizable} if for each
$Y \in \cat T$ the natural map 
$\mu: \mathscr{H}om(X, \mathbf{1}) \otimes Y \rar \mathscr{H}om(X,  Y)$ is an isomorphism.
	A tensor triangulated category $(\cat T,  \otimes , \mathbf{1})$ is \emph{rigidly compactly
	generated} if the following conditions hold:
	\begin{itemize}
		\item[i.] \cat T is compactly generated;
		\item[ii.] $\mathbf{1}$ is compact;
		\item[iii.] every compact object is rigid. 
	\end{itemize}
\end{definition}

Let $\cat T$ be a rigidly compactly generated tensor triangulated category.  For such a 
triangulated category the tensor product on $\cat T$ restricts to $\cat T^c$ therefore 
$(\cat T^c,  \otimes , \mathbf{1})$ is also a tensor triangulated category.  Suppose $\cat T$ is 
given with a preaisle $\aisle T$ satisfying the condition of Definition \ref{definition} then so 
does the preaisle $\cat T^c \cap \aisle T$ of $\cat T^c$.  So we can define \tensor preaisles and 
\tensor copreaisles of $\cat T^c$ with respect to $\cat T^c \cap \aisle T$.

\begin{lemma}
\label{preaisle and copreaisle}
	Let $\cat T$ be a rigidly compactly generated tensor triangulated category given with a 
	preaisle $\aisle T$ satisfying the condition of Definition \ref{definition}.
	
	Then, there is a one-to-one  correspondence between the set of \tensor preaisles and the set 
	of \tensor copreaisles of $\cat T^c$.
\end{lemma}

\begin{proof}

Let $\cat U$ be a full subcategory of $\cat T ^c$.  We denote by $\cat U^*$ the full 
subcategory \[ \cat U ^* = \{ X \in \cat T^c \mid X \cong \mathscr{H}om ( Y,  \mathbf{1}) \text{ 
for some } Y \in \cat U \}. \]

The assignment $\cat U \mapsto \cat U^*$ induces an equivalence between the preaisles and 
copreaisles of $\cat T ^c$; see \cite[Lemma 4.9]{SP16}.  We only need to show that the 
above assignment preserves the tensor condition.

Let $\cat U$ be a preaisle of $\cat T^c$, $X \in \cat U^*$ and $T \in \cat T^c \cap \cat T ^{\leq 
0}$.  We have \[ \mathscr{H}om(T, X) \cong \mathscr{H}om (T,  \mathscr{H}om (Y, \mathbf{1}))
\cong \mathscr{H}om (T \otimes Y , \mathbf{1}).\]

If we assume $\cat U$ is a \tensor preaisle then $T \otimes Y \in \cat U$.  Hence  
$ \mathscr{H}om(T, X) \in \cat U^*$,  this proves $\cat U^* $ is a \tensor copreaisle.  

Now,  suppose $\cat U$ is a copreaisle.  Let $X \in \cat U^*$ and $T \in \cat T^c \cap \cat T ^{\leq 0}$.  Then,

\begin{align*}
 T \otimes X
 & \cong T \otimes \mathscr{H}om(Y,\mathbf{1})\\
 & \cong \mathscr{H}om(Y,T)\\
 & \cong \mathscr{H}om( Y \otimes \mathscr{H}om(T, \mathbf{1}),  \mathbf{1})\\
 & \cong \mathscr{H}om( \mathscr{H}om(T,Y), \mathbf{1}).\\
\end{align*}

If we assume $\cat U$ is a \tensor copreaisle then $\mathscr{H}om(T,Y) \in \cat U$.  We get,
$T \otimes X \in \cat U^*$,  which proves $\cat U^*$ is a \tensor preaisle.

\end{proof}

\begin{definition}
A preaisle \cat U is \emph{compactly generated} if $\cat U = \preaisle S$ for a set of compact 
objects \cat S.
\end{definition}

\begin{definition}
\label{*}
	We say a triangulated category $\cat T$ has the property $(*)$ if:
	\begin{itemize}
		\item[i.] $\cat T$ is rigidly compactly generated;
		\item[ii.] $\cat T$ has a preaisle $\aisle T$ satisfying $\aisle T \otimes \aisle T \subset 
		\aisle T	\text{ and }\mathbf{1} \in \aisle T$;
		\item[iii.] $\aisle T$ is compactly generated,  that is,  $\aisle T = \langle \cat T^c \cap
		 \aisle T \rangle ^{\leq 0}$.
	\end{itemize}	 
\end{definition}

For a triangulated category $\cat T$ having the property $(*)$, we define a weaker notion than 
\tensor (co)preaisle.

\begin{definition} Let $\cat T$ have the property $(*)$.

A preaisle  $\cat U$ of $\cat T$ is a \emph{$\otimes ^c$-preaisle} if for any $T \in \cat T^c 
\cap \aisle T$ and $U \in \cat U$ we have $T \otimes U \in \cat U$.  Similarly a copreaisle 
$\cat X$ of $\cat T$ is a \emph{$\otimes ^c$-copreaisle} if for any $T \in \cat T^c \cap \aisle T
$ and $X \in \cat X$ we have $\mathscr{H}om(T,  X) \in \cat X$. 

A t-structure $(\cat U,  \cat V)$ on $\cat T$ is a \emph{$\otimes ^c$-t-structure} if $\cat U$ is 
a $\otimes ^c$-preaisle and a weight structure $(\cat X, \cat Y)$ on $\cat T$ is a 
\emph{$\otimes ^c$-weight structure} if $\cat X$ is a $\otimes ^c$-copreaisle. 
\end{definition}

\begin{remark} 
\label{weak}
This weaker notion gives something new only for preaisles(resp. copreaisles) which are 
not cocomplete(resp. complete) since by  Lemma  \ref{tensor generator 1} it can easily be 
observed that for $\cat T$ having the property $(*)$: (i) every cocomplete 
$\otimes ^c$-preaisles of $\cat T$ is a $\otimes$-preaisle of $\cat T$,  and (ii) every complete 
$\otimes ^c$-copreaisle of $\cat T$ is a \tensor copreaisle of $\cat T$.
\end{remark}

With this weaker notion,  we now prove the tensor analogue of Theorem \ref{SP theorem}.

\begin{theorem}
\label{t-structure}
	Let $\cat T$ have the property $(*)$ (see Definition \ref{*}) and $\cat T = 
	\mathbb{D}(e)$, where $\mathbb{D}$ is a stable derivator such that for each
	small category $I$,  $\mathbb{D}(I)$ has all small coproducts.  Then,
	
	\begin{itemize}
		\item[i.] There is a bijective correspondence between the set for compactly generated 
		tensor t-structures of $\cat T$ and the set of thick \tensor preaisles of $\cat T^c$
		given by
		\[ (\cat U, \cat V) \mapsto \cat U \cap \cat T^c \]
		\[ \cat P \mapsto (^{\perp} (\cat P ^{\perp}),  \Sigma \cat P ^{\perp}).\]
		
		\item[ii. ]There is a bijective correspondence between the set for compactly generated 
		$\otimes ^c$-weight structures of $\cat T$ and the set of thick \tensor copreaisles of $
		\cat T^c$ given by
		\[ (\cat X, \cat Y) \mapsto \cat X \cap \cat T^c \]
		\[ \cat P \mapsto (^{\perp} (\cat P ^{\perp}),  \Sigma^{-1} \cat P ^{\perp}).\]
	\end{itemize}
\end{theorem}

Before proving the theorem, we will make some comments about the lack of symmetry in the
above statement:
\begin{remark}
\label{remark 1}
	 It is easy to observe that given a subcategory $\cat P$ the subcategory
	${}^{\perp} (\cat P ^{\perp})$ is always cocomplete,  that is,  closed under coproducts.  
	Therefore,  by Remark \ref{weak} saying ${}^{\perp} (\cat P ^{\perp})$ is a $\otimes ^c$-
	preaisle of $\cat T$ is equivalent to saying it is a \tensor preaisle.
	
\end{remark}

\begin{remark}
\label{remark 2}
	As $\cat T$ has coproducts and is compactly generated, by Brown representability it has
	products.  However,  it is not clear to us whether ${}^{\perp} (\cat P ^{\perp})$ is closed 
	under products, hence unlike the preaisle case, we can not claim it is \tensor copreaisle.
\end{remark}

\begin{proof}[Proof of Theorem \ref{t-structure}]
	Part(i).   It has already been shown in \cite[Theorem 4.5(i)]{SP16} that the above assignments 
	are bijections between the set of compactly generated t-structures of $\cat T$ and the set 
	of thick preaisles of $\cat T^c$.  We only need to show that the assignments preserve the 
	tensor conditions. 
	   
	From the definition of \tensor preaisle, it is easy to observe that if $\cat U$ is a 
	\tensor preaisle of $\cat T$ then $\cat U \cap \cat T^c$ is a \tensor preaisle of $\cat T^c$.
	Suppose $\cat P$ is a \tensor preaisle of $\cat T^c$.  Since 
	$^{\perp} (\cat P ^{\perp})$ is a cocomplete preaisle of $\cat T$ by Remark \ref{weak} 
	it is enough to show $^{\perp} (\cat P ^{\perp})$ is a $\otimes ^c$-preaisle of $\cat T$.  
	Let $\cat B = \{ X \in$ $  ^{\perp} (\cat P ^{\perp})  \mid (\cat T^c \cap \aisle T) \otimes X
	\subset$ $ ^{\perp} (\cat P ^{\perp}) \}$.  We note that $\cat B$ is a cocomplete preaisle
	containing $\cat P$.  Since $^{\perp} (\cat P ^{\perp})$ is the smallest cocomplete preaisle 
	containing $\cat P$ by \cite[Lemma 1.9]{DS22} we get $\cat B = $ $^{\perp} (\cat 
	P^{\perp})$.  
	
	Part(ii).  In view of \cite[Theorem 4.5(ii)]{SP16},  again we only need to show that the 
	assignments 
	preserve the appropriate tensor conditions.  If $\cat X$ is a $\otimes ^c$-copreaisle of 
	$\cat T$ then it is easy to observe that $\cat X \cap \cat T^c$ is a \tensor copreaisle of 
	$\cat T^c$.  Suppose $\cat P$ is a \tensor copreaisle of $\cat T^c$ we need to show that 
	$^{\perp} (\cat P ^{\perp})$ is a $\otimes ^c$-copreaisle of $\cat T$.  By 
	\cite[Theorem 3.7]{SP16} an object $A \in \cat T$ belongs to $ ^{\perp} (\cat P ^{\perp})$ 
	if and only if $A$ is a summand of a homotopy colimit of a sequence
	\begin{center}

	\begin{tikzcd}
	 &0 = Y_0 \arrow[r, "f_0"] & Y_1 \rar ["f_1"] & Y_2 \rar ["f_2"] &\cdots 
	\end{tikzcd}
	\end{center}
	
	where each $f_i$ occurs in a triangle $ Y_i \rar Y_{i+1} \rar S_i \rar \Sigma Y_i$ with
	$S_i \in \text{Add }\cat P$.  First, we observe that for any compact object $T$ the functor
	$\mathscr{H}om( T, -)$ preserves small coproducts therefore $\mathscr{H}om( T, -)$ takes 
	homotopy sequences to homotopy sequences.  Since $\cat P$ is \tensor copreaisle of
	$\cat T^c$ for any $T \in \cat T^c \cap \aisle T$ we have $\mathscr{H}om(T,  S_i) \in 
	\text{Add } 
	\cat P$.  Thus applying \cite[Theorem 3.7]{SP16} again we get $\mathscr{H}om( T, A) \in $
	 $^{\perp} (\cat P ^{\perp}).$

\end{proof}

\section{The classification theorem for weight structures}

Let $X$ be a Noetherian separated scheme.  $\deriveq X$ denotes the derived category of  
complexes of quasi coherent $\CO _X$-modules.
The derived category $( \deriveq X, \otimes ^{L} _{\CO _X},\CO_X)$ is a tensor triangulated 
category with the derived tensor product $\otimes ^{L} _{\CO _X}$  and the structure sheaf 
$\CO_X$ as the unit.  The full subcategory of complexes whose cohomologies vanish in 
positive degree $\mathbf{D}^{\leq 0}(\mathrm{Qcoh} X)$ is a preaisle of $\deriveq X$ 
satisfying the conditions of Definition \ref{definition}.  \emph{We define the \tensor preaisles
and \tensor copreaisles of $\deriveq X$ with respect to $\mathbf{D}^{\leq 0}(\mathrm{Qcoh} 
X)$}.  Similarly,  the \tensor preaisles and \tensor copreaisles of $\perfect X$ are defined 
with respect to $ \mathrm{Perf} ^{\leq 0} (X)$.  Note that $\deriveq X$ has the property $(*)$ 
(see Definition \ref{*}),  so we can define $\otimes ^c$-(co)preaisles of $\deriveq X$.

\begin{definition}
	
\label{D Thomason subset}
	A subset $Z$ is a \emph{specialization closed} subset of $X$ if for each $x \in Z$ the 
	closure of the singleton set $\{x\}$ is contained in $Z$, that is, $\bar{\{x\}} \subset Z$.  Note 
	that a specialization closed subset is a union of closed subsets of $X$.  

	A subset $Y$ is a \emph{Thomason} subset of $X$ if $Y = \bigcup _{\alpha} Y_{\alpha}$ is a 
	union of closed subsets $Y_{\alpha}$ such that $X \setminus Y_{\alpha}$ is quasi compact. 
	Note that if $X$ is Noetherian then the two notions coincide.

\end{definition}

\begin{definition}
	A \emph{Thomason filtration} of $X$ is a map $\phi : \Z \rar 2^X$ such that 
	$\phi(i)$ is a Thomason subset of $X$ and $\phi(i) \supset \phi (i+1)$ for all $i \in \Z$.
\end{definition}

In our earlier work we have mentioned without proof (see \cite[Remark 4.13]{DS22}) about the 
following result,  here we explicitly state it for future reference.  This is a generalization 
of Thomason's classification  \cite[Theorem 3.15]{Thomason} of \tensor ideals to \tensor 
preaisles of $\perfect X$, for separated Noetherian scheme $X$.

\begin{proposition}
\label{tensor preaisle}
Let $X$ be a separated Noetherian scheme.  The assignment sending a Thomason
filtration
	$ \phi$ to $ \cat S_{\phi} = \{ E \in \perfect X \mid \Supp (H^i(E)) \subset \phi (i) \} $
	 provides a one-to-one correspondence between the following sets:

\begin{itemize}
	\item[i.] Thomason  filtrations of $X$;
	\item[ii.] Thick \tensor preaisles of $\perfect X$.
	
\end{itemize}

\end{proposition}

\begin{proof}
	In \cite[Theorem 4.11]{DS22} we have shown that sending $\phi$ to 
		\begin{center}
			$ \cat U_{\phi} = \{ E \in \deriveq X \mid \Supp (H^i(E)) \subset \phi (i) \} $
		\end{center}
	 provides a bijection between the set of Thomason filtrations and the set of 
	 compactly generated tensor t-structure of $\deriveq X$.  From part(i) of Theorem 
	\ref{t-structure}, we conclude that the above assignment provides a bijection between  
	 Thomason filtrations of $X$ and thick \tensor preaisles of $\perfect X$. 
\end{proof}

\begin{theorem}
\label{tensor weight structure}
Let $X$ be a separated Noetherian scheme.  There is a one-to-one correspondence between
the following sets:

\begin{itemize}
	\item[i.]  Thomason Filtrations of $X$;
	\item[ii.] Compactly generated $\otimes ^c$-weight structures of 
	$\deriveq X$.
	
The assignment is given by 
\[\phi \mapsto (\cat A_{\phi} , \cat B_{\phi})\]
	where\\
	$\cat B _{\phi} = \{ B \in \deriveq X \mid \Hom (\CO_X,  S \otimes ^{L} _{\CO _X} B) = 0
	\text{ for all } S \in \cat S_\phi \},$\\
	$\cat S_\phi = \{S \in \perfect X \mid \Supp (H^iS) \subset \phi (i) \},$ and\\
	$\cat A_\phi = \{ A \in \deriveq X \mid \Hom (A , B) = 0 \text{ for all } B \in \cat B_{\phi} \}.$
\end{itemize}

\end{theorem}

\begin{proof}

	Let $\phi$ be a Thomason filtration of $X$.  By Proposition \ref{tensor preaisle} 
	we know $\phi \mapsto \cat S_{\phi}$ is a bijection.  Now sending $\cat S_{\phi}$ to 
	$\cat S_{\phi} ^*$ is again a bijection by Lemma \ref{preaisle and copreaisle}.  Since
	$\cat S_{\phi} ^*$ is a \tensor copreaisle of $\perfect X$,  the assignment $
	\cat S_{\phi} ^* \mapsto ( ^{\perp}((\cat 
	S_{\phi} ^*) ^{\perp}) ,(\cat S_{\phi} ^*)^{\perp})$ is a bijection by Theorem 
	\ref{t-structure}.  We only need to show that $\cat B_{\phi} =  (\cat S_{\phi} ^*)^{\perp}$ 
	which is the consequence of the tensor-hom adjunction.   

\end{proof}

\section{In the case of projective line}

In this section, we will specialize to the case of projective line $\mathbb{P}^1_k$ over a field 
$k$.  By the results of earlier sections, classifying compactly generated tensor t-structures of 
$\mathbf{D}(\mathrm{Qcoh \hspace{1mm}}\mathbb{P}^1_k)$ is equivalent to classifying thick 
\tensor preaisles of  $\perfect {\mathbb{P} ^1 _k}$.  For any smooth Noetherian scheme $X$ 
the inclusion functor from $\perfect X$ to the derived category of bounded complexes of 
coherent sheaves $\mathbf{D}^b(\mathrm{Coh\hspace{1mm}} X)$ is an equivalence.   
Therefore,  we restrict our attention to $\deriveb$.  Note that we define \tensor preaisles of 
$\deriveb$ with respect to the standard preaisle

\[ \mathbf{D}^{b, \leq 0} (\mathrm{Coh\hspace{1mm}} \mathbb{P}^1_k ) \colonequals \{ E \in 
\deriveb \mid H^i(E) = 0 \hspace{1mm} \forall i > 0\}.\]

\begin{lemma}
\label{ch lemma}
	A thick preaisle $\cat A$ of $\deriveb$ is a \tensor preaisle if and only if 
	\[\CO (-1) \otimes \cat A \subset \cat A.\]
\end{lemma}

\begin{proof}
	Suppose $\cat A$ is a \tensor preaisle then $\CO (-1) \otimes \cat A \subset \cat A$ is 
	true by definition.   Conversely,  suppose $\cat A$ is a preaisle of $\deriveb$.  Take 
	$\cat B \colonequals \{ B \in \mathbf{D}^{b, \leq 0} (\mathrm{Coh\hspace{1mm}} 
	\mathbb{P}^1_k ) \mid B \otimes \cat A \subset \cat A\}$. From our assumption, we have 
	$\CO (-1) \in \cat B$.  It is now easy to see that for every $n \geq 0$ we have $
	\CO (-n) \in \cat B$.
	
	As $\mathrm{Coh} \hspace{1mm}\mathbb{P}^1 _k$ has homological 
	dimension one,  every complex of $\deriveb$ is quasi isomorphic to the direct sum of its 
	cohomology sheaves,  see \cite[Proposition 6.1]{Rudakov}.  Also, every coherent sheave over 
	$\mathbb{P}^1_k$ is the direct sum 
	of line bundles and torsion sheaves.  Since $\cat B$ is a preaisle,  to show
	$\cat B = \mathbf{D}^{b,  \leq 0} (\mathrm{Coh\hspace{1mm}} \mathbb{P}^1_k )$ it is
	enough to show that $\cat B$ contains all line bundles and torsion sheaves. 
	
	For any $m \geq 0$ consider the following triangle coming from the corresponding short 
	exact sequence in $\mathrm{Coh} \hspace{1mm} \mathbb{P}^1 _k$; see for instance 
	\cite[Equation 6.3]{Rudakov},
	\[ \CO (-2) ^{\oplus (m +1)} \longrightarrow \CO (-1)^{\oplus (m+2)} \longrightarrow \CO (m) 
	\longrightarrow \CO (-2)[1].\]
	Since $\cat B$ is closed under extension and positive shifts we have $\CO (m) \in \cat B$. 
	
	Next,  for any indecomposable torsion sheave of degree $d$ say $ T_x$, which is 
	supported on a closed point $x \in \mathbb{P}^1_k$,  consider the following triangle coming 
	from the corresponding short exact sequence in 
	$\mathrm{Coh} \hspace{1mm} \mathbb{P}^1 _k$; see \cite[Equation 6.5]{Rudakov},
	\[  \CO (-2)^{\oplus d} \longrightarrow \CO (-1) ^{\oplus d} \longrightarrow T_x 
	\longrightarrow \CO (-2)[1].\]
	
	Again using the fact that $\cat B$ is closed under extension and positive shifts we have 
	$T_x \in \cat B$.

\end{proof}

Recall that for a set of objects $\cat S$ of $\cat T$ we denote the smallest cocomplete
preaisle containing $\cat S $ by $\preaisle S$.  If $\cat T$ does not have coproducts, for 
instance $\deriveb$, we denote $\preaisle S$ to be the smallest preaisle containing $\cat S$. 
Similarly we denote $\langle \cat S \rangle ^{\geq 0}$ to be the smallest copreaisle containing
$\cat S$.  Also recall that for any subcategory $\cat U$  we denote $\cat U^*$ the full 
subcategory \[ \cat U ^* = \{ X \in \cat T \mid X \cong \mathscr{H}om ( Y,  \mathbf{1}) \text{ 
for some } Y \in \cat U \}. \]

\begin{example}
	For a fixed $n \in \Z$ we denote 
	\begin{align*}
	& \cat B_n \colonequals \langle \CO (n) \rangle ^{\leq 0}; \text{ and} \\
	& \cat C_n \colonequals \langle \CO (n) ,  \CO (n+1) \rangle ^{\leq 0}.\\
	\end{align*}
	Using Lemma \ref{ch lemma}, we can check that $\cat B_n$ and $\cat C_n$  are not 
	\tensor preaisles of $\deriveb$.  Similarly,  $\cat B_n ^*$ and $\cat C_n ^*$ 
	provide examples of copreaisles of $\deriveb$ which are not \tensor copreaisles.  This can
	be observed using Lemma \ref{preaisle and copreaisle}. 
\end{example}

Recall that a Thomason filtration of $X$ is a map $\phi : \Z \rar 2^X$ such that 
$\phi(i)$ is a Thomason subset of $X$ and $\phi(i) \supset \phi (i+1)$ for all $i \in \Z$.  We say
$\phi $ is \emph{type-1} if $\bigcup _i \phi (i) \neq X$; and we say $\phi$ is \emph{type-2} if 
$\bigcup _i \phi (i) = X$ but not all $\phi(i) =X$.

Let $x \in \mathbb{P} ^1 _k$ be a closed point.  We denote the simple torsion sheaf supported 
on $x$ by $k(x)$.  Now, we give an explicit description of the \tensor preaisles of $\deriveb$ in 
terms of simple torsion sheaves and line bundles.

\begin{proposition}
\label{Main}
	Any proper thick \tensor preaisle of $\deriveb$ is one of the following forms:
	
	\begin{itemize}
\item[i.] $\langle k(x)[-i] \mid x \in \phi (i) \rangle ^{\leq 0};$

\end{itemize}

where $\phi$ is a type-1 Thomason filtration of $\mathbb{P}^1_k$.
\begin{itemize}

\item[ii.] $\langle \CO (n) [-i_0],  k(x)[-i] \mid \forall n \in \Z \text{ and } x \in \phi (i)  \rangle ^{\leq 0}$ 

\end{itemize}

	where $\phi$ is a type-2 Thomason filtration of $\mathbb{P}^1_k$ and $i_0$ a fixed 
	integer.
\end{proposition}

\begin{proof}
	
	Suppose $\cat A$ is a thick \tensor preaisle of $\deriveb $.  By Proposition \ref{tensor 
	preaisle} there is a unique Thomason filtration $\phi$ such that
	\[ \cat A = \{ E \in \deriveb \mid \Supp H^i(E) \subset \phi (i) \}.\]
	
	Since $\mathrm{Coh} \hspace{1mm}\mathbb{P}^1 _k$ has homological dimension one,  every 
	complex of $\deriveb$ is quasi isomorphic to the direct sum of its cohomology sheaves. 
	Therefore,  we can write $\cat A$ in terms of coherent sheaves alone,
	\[ \cat A = \langle F[-i] \mid F \in \mathrm{Coh} \hspace{1mm} \mathbb{P} ^1 _k \text{ and } 
	\Supp F \subset \phi (i) \rangle ^{\leq 0} . \]
	
	\textit{Case 1.} ( $\phi$ is type-1 ) Note that $\phi (i) \subsetneq \mathbb{P}^1 _k $ 
	for all $i$.
	Every coherent sheave over $\mathbb{P}^1_k$ is the direct sum of line bundles and torsion
	sheaves.  Since the support of any line bundle is whole $\mathbb{P}^1 _k$.  In this case,
	$\cat A$ only contains torsion sheaves.  As torsion sheaves can be generated 
	by
	simple torsion sheaves we have,
	
	\[ \cat A = \langle k(x)[-i] \mid x \in \phi (i) \rangle ^{\leq 0}.\]
	
	\textit{Case 2.}( $\phi$ is type-2 ) Since $\bigcup _i \phi (i) = X$  there is an integer $i_0$
	such that $\phi(i_0)$ contains the generic point of $\mathbb{P}^1 _k$.  
	We can take $i_0$ to be the largest such integer.  Observe that $\phi (i) = \mathbb{P}^1 _k $ 
	for all $i \leq i_0$ and $\phi(i_0+1)  \subsetneq \mathbb{P}^1 _k$.  Here we can check that
	\[ \cat A = \langle \CO (n) [-i_0],  k(x)[-i] \mid \forall n \in \Z \text{ and }  x \in \phi (i) \rangle 
	^{\leq 0}. \]
	
\end{proof}

Next, we will show which of these \tensor preaisles of $\deriveb$ are t-structures on 
$\deriveb$.  First, we prove a few lemmas.

\begin{lemma}
\label{delta lemma}
	Let $A \in \mathrm{Coh} \hspace{1mm}\mathbb{P}^1 _k$ be a torsion sheaf and 
	$\mathscr{L}$ be a line bundle.  Let $\delta : A \rar\mathscr{L}[1]$ be any map in $\deriveb$.   
	Then,  $cone(\delta) \notin A^{\perp}$.
\end{lemma}

\begin{proof}
	 As we know $\Hom (A , \mathscr{L}[1]) \cong \mathrm{Ext}^1 (A,  \mathscr{L})$,  a map 
	 $\delta : A  \rar \mathscr{L}[1]$ corresponds to an element of the group 
	 $\Ext ^1(A,\mathscr{L})$.  By abuse of notation, we denote the corresponding element in 
	 $\Ext ^1(A,\mathscr{L})$ by $\delta$.
	
	Now we take the short exact sequence corresponding to $\delta \in \Ext^1(A, \mathscr{L})$, 
	say
		\begin{center}
		\begin{tikzcd}
	 		0 \rar & \mathscr{L} \rar & B \rar  & A \rar & 0.
		 \end{tikzcd}
		\end{center}
	This gives rise to a distinguished triangle 
	\begin{center}
	
	\begin{tikzcd}
	 \mathscr{L} \rar & B \rar & A \rar  [" \delta "] & \mathscr{L} [1]
	\end{tikzcd}
	\end{center}
	
	Hence,  $cone(\delta) \cong B[1]$.  From the short exact sequence, we observe that $B$ can 
	not be a torsion sheaf,  so it must have a torsion free summand.  Therefore, $\Hom (A, B[1]) = 
	\Ext ^1 ( A, B) \neq 0$. 
	
\end{proof}

Recall that a preaisle $\cat A$  is an aisle if  $(\cat A, \cat A^{\perp} [1])$ a t-structure.  

\begin{lemma}
\label{one step}
	Let $\cat A$ be a \tensor preaisle of $\deriveb$ and $\phi$ its corresponding 
	Thomason filtration.  If $\cat A$ is an aisle and
	$\phi (i) \neq  \emptyset$  for some $i$,  then $\phi (i-1) = \mathbb{P}^1 _k$.
\end{lemma}

\begin{proof}
Without loss of generality, we may assume $i = 0$.  If $\phi (0) = \mathbb{P} ^1 _k$ then
$\phi (-1) =  \mathbb{P} ^1 _k$ and there is nothing to prove.  Now suppose 
$\phi (0) \subsetneq \mathbb{P} ^1 _k$ then there is
a closed point $x \in \phi (0)$.  We will prove our claim by showing a contradiction. 

Let $\mathscr{L}$ be a line bundle on $\mathrm{Coh} \hspace{1mm}\mathbb{P}^1 _k$. If 
$\phi (-1) \neq \mathbb{P}^1 _k$ then $\mathscr{L}[1] \notin \cat A$.  Since $\mathrm{Ext}^1
  ( k(x),  \mathscr{L}) \neq 0$ we also have $\mathscr{L}[1] \notin \cat A ^{\perp}$.  Now, as
  $\cat A$ is given to be an aisle we must have a t-decomposition of $\mathscr{L}[1]$. 
  But Lemma \ref{delta lemma} says such a decomposition is not possible.
\end{proof}

\begin{definition}
	We say $\phi$ is a \emph{one-step} Thomason filtration of $\mathbb{P}^1 _k$ 
	if there is an integer $i_0$ and a Thomson subset $Z_{i_0}$ such that
	\begin{align*}
	\phi(j)  & = \mathbb{P}^1 _k  \text{ \hspace{0.5mm} if } j < i_0;\\
	& = Z_{i_0} \text{ \hspace{0.5mm} if } j = i_0;\\
	& = \emptyset \text{ \hspace{0.5mm} if } j > i_0.\\
	\end{align*}
\end{definition}

\begin{proposition}
\label{aisle}
	A \tensor preaisle of $\deriveb$ is an aisle if and only if the corresponding 
	 Thomason filtration is a one-step filtration. 
\end{proposition}

\begin{proof}
	If $\cat A$ is a \tensor preaisle which is also an aisle then by Lemma \ref{one step} the 
	corresponding filtration is a one-step filtration.Conversely,  suppose the filtration is one step, 
	we will show that  every complex of $\deriveb$ can be decomposed into a triangle where 
	the first term is in $\cat A$ and the third term is in $\cat A ^{\perp}$.  Without loss of 
	generality we may assume the one step occurs at $i_0 = 0$,  and $\phi (0) = Z_0$ a 
	Thomason subset.  
	
	If $Z_0 =  \mathbb{P}^1 _k$,  then $\cat A = \mathbf{D}^{b, \leq 0} (\mathrm{Coh
	\hspace{1mm}} \mathbb{P}^1_k )$ and we get standard t-structure.  Now, suppose
	 $Z_0 \neq \mathbb{P}^1 _k$.  Since the filtration is one step we only need to
	 show sheaves at degree zero have t-decompositions,  all other shifted sheaves have 
	 obvious t-decompositions.  The functor $\Gamma_{Z_0}(-)$ gives a t-decomposition of 
	 sheaves at degree zero.

\end{proof}

Next,  we give an explicit description of \tensor copreaisles of $\deriveb$ in terms of simple torsion sheaves and line bundles.

\begin{proposition}
	Any proper thick \tensor copreaisle of $\deriveb$ is one of the following forms:
	
	\begin{itemize}
\item[i.] $\langle k(x)^* [i] \mid x \in \phi (i) \rangle ^{\geq 0};$

\end{itemize}

where $\phi$ is a type-1  Thomason filtration of $\mathbb{P}^1_k$.
\begin{itemize}

\item[ii.] $\langle \CO (n) [i_0],  k(x)^*[i] \mid \forall n \in \Z \text{ and } x \in \phi (i)  \rangle 
^{\geq 0}$ 

\end{itemize}
	
where $\phi$ is a type-2  Thomason filtration of $\mathbb{P}^1_k$ and $i_0$ a fixed 
integer.
\end{proposition}

\begin{proof}
	By the proof of Lemma \ref{preaisle and copreaisle},  we know that every \tensor 
	copreaisle of $\deriveb$ is of the form $\cat A^*$ where $\cat A $ is \tensor preaisle.  Now
	using the description given in Proposition \ref{Main} we conclude our result.

\end{proof}

The trivial \tensor copreaisles $\deriveb$ and $0$ give rise to tensor weight structures on 
$\deriveb$.  In contrast to the case of t-structures (see Proposition \ref{aisle}),  the next result
shows that, there are no other tensor weight structures on $\deriveb$.

\begin{proposition}
	The trivial weight structures  are the only tensor weight structures on $\deriveb$.
\end{proposition}

\begin{proof}
	Suppose $\cat A$ is a \tensor copreaisle of $\deriveb$ which induces a weight structure.  We 
	claim that $\cat A$ can not be a copreaisle containing only torsion sheaves.  Indeed,  by 
	Lemma \ref{delta lemma} any line bundle $\mathscr {L}[1]$ can not have a weight 
	decomposition, so $\cat A$ must contain line bundles upto shifts.  If $\mathscr{L}[i] \in 
	\cat A$ for some $\mathscr{L}$, then for any line bundle $\mathscr{M}$,  $\mathscr{M}[i] = 
	\mathscr{H}om( \check{\mathscr{M}}\otimes \check{\mathscr{L}},  \mathscr{L}[i]) \in 
	\cat A$.
	
	Now, there are two cases: (1) either there is an integer $i$ such that
	$\mathscr{L}[i] \in \cat A$  and $\mathscr{L}[i+1] \notin \cat A$,  or (2) there is no such $i$
	and $\cat A$ contains all line bundles and their shifts.
	
	\textit{Case 1.} Suppose there is an integer $i$, then without loss of generality we can 
	assume $i = 0$.  By our assumption, for any line bundle $\mathscr{M}$,  $\mathscr{M} [1] 
	\notin \cat A$.  Since by tensor condition $\cat A$ contains all line bundles we can choose 
	$\mathscr{L}$ such that $\Ext^1 (\mathscr{L}, \mathscr{M}) \neq 0$, therefore, 
	$\mathscr{M} [1] \notin \cat A^\perp$.  Then by a similar argument as in Lemma 
	\ref{delta lemma}, $\mathscr{M} [1] $ can not have 
	a weight decomposition (since any distinguished triangles with $ \mathscr{M} [1]$ in the 
	middle will result in a third term having a summand isomorphic to $ \mathscr{M} [1]$,  and
	$ \mathscr{M} [1]$
	is not in $\cat A^{\perp}$).  
	Therefore, $\cat A$ can not induce a weight structure. 
	
	\textit{Case 2.} Suppose $\cat A$ contains all line bundles and their shifts.  In particular,
	it contains $\CO (-1)$, $\CO (-2)$ and all their shifts.  Now,  for any indecomposable torsion 
	sheave of degree $d$ say $ T_x$ supported on a closed point $x \in \mathbb{P}^1_k$,  
	consider the following triangle coming from the corresponding short exact sequence in 
	$\mathrm{Coh} \hspace{1mm} \mathbb{P}^1 _k$,
	\[  \CO (-2)^{\oplus d} \longrightarrow \CO (-1) ^{\oplus d} \longrightarrow T_x 
	\longrightarrow \CO (-2)[1].\]
	
	As $\cat A$ is closed under extensions, $T_x \in \cat A$.  This proves $\cat A$ 
	contains all torsion sheaves and their shifts.  Therefore, $\cat A$ must be equal to 
	$\deriveb$.

\end{proof}

\section*{Acknowledgement}

The authors are grateful for the peaceful work environment and the assistance of the support 
staff of HRI, Prayagraj.  The second author is supported in part by the INFOSYS scholarship. 

\bibliographystyle{alpha}
\bibliography{DS2211ref}

\end{document}